\documentclass[a4paper,reqno]{amsart}
\usepackage[margin=1.0in]{geometry}
\usepackage[utf8]{inputenc}
\usepackage{xcolor,amsmath,amssymb,indentfirst,multirow,hhline,esvect,centernot,tensor,enumitem,amsthm,tikz,fancyhdr,latexsym,listings,graphicx,wrapfig,thmtools,setspace,wrapfig}
\usepackage[toc,page]{appendix}
\usepackage[colorlinks = true, linkcolor = blue, citecolor = carmine, anchorcolor = blue, urlcolor = black]{hyperref}
\usetikzlibrary{cd}
\newcommand{\ar}[1]{[\href{https://arxiv.org/abs/#1}{#1}]}
\definecolor{carmine}{rgb}{0.59, 0.0, 0.09}

\makeatletter
\newcommand*\bigcdot{\mathpalette\bigcdot@{.5}}
\newcommand*\bigcdot@[2]{\mathbin{\vcenter{\hbox{\scalebox{#2}{$\m@th#1\bullet$}}}}}
\makeatother

\declaretheorem[name=Theorem, parent=section]{theorem}
\declaretheorem[name=Proposition, sibling=theorem]{prop}
\declaretheorem[name=Corollary, sibling=theorem]{cor}
\declaretheorem[name=Lemma, sibling=theorem]{lemma}
\declaretheorem[name=Definition, style=definition, sibling=theorem]{defn}
\declaretheorem[name=Example, style=definition, sibling=theorem]{example}
\declaretheorem[name=Remark, style=definition, sibling=theorem]{remark}

\newcommand{\pin}{\mathrm{Pin}}
\newcommand{\on}{\operatorname}
\newcommand{\gm}{{\mathcal G}}

\title{ \textbf{On the generalised Lie derivative of (s)pinor fields}}

\author{Frederik \v Dalak}
\address{Mathematical Institute, Faculty of Mathematics and Physics, Charles University, Prague 186 75, Czech Republic}
\email{frederik.dalak@matfyz.cuni.cz}
\author{Fridrich Valach}
\address{Mathematical Institute, Faculty of Mathematics and Physics, Charles University, Prague 186 75, Czech Republic}
\email{fridrich.valach@matfyz.cuni.cz}

\begin{document}

\begin{abstract}
  Although the action of diffeomorphisms on pinor (or spinor) fields is geometrically unambiguous, the definition of a natural Lie derivative is more subtle and requires some understanding of the geometry of the space of metrics at a point. In this work we revisit the classical construction of Kosmann and Bourguignon--Gauduchon and develop the corresponding theory in generalised geometry.
\end{abstract}

\maketitle

\section{Introduction}
\subsection{Lie derivative of (s)pinors}  
  Given $M$ a smooth manifold, any diffeomorphism $\phi\colon M\to M$ lifts naturally to an automorphism of all associated tensor bundles. Thus, pulling back a vector field along $\phi$ one again gets a vector field; pulling back a metric $g$ one obtains a metric $\phi^*g$, and so on. If the metric $g$ admits a pin lift, any associated pinor (or in an oriented case spinor) field $\psi$ can again be naturally pulled back, resulting in a pinor (or spinor) field $\phi^*\psi$ which, however, is associated with the metric $\phi^*g$.
  
  This in particular implies that the naive definition of the Lie derivative of any pinor field $\psi$ along a vector field with flow $\phi_t$ as
  \[\left.\frac{d}{dt}\right|_{t=0}\phi_t^*\psi\]
  makes no sense in general, as it is not meaningful to take a difference of sections of different vector bundles.\footnote{The definition however makes perfect sense when the vector field is Killing, since in this case the pulled-back pinor field remains the section of the same bundle.}
  
  This can be conveniently visualised as follows. Fixing a manifold $M$, the space of all possible pairs $(g,\psi)$, consisting of a metric and a pinor field, is naturally an (infinite-dimensional) vector bundle $\mathcal B\to\mathcal M$ over the space of metrics, with the fibre given by the space of sections of the associated pinor bundle. The action of any diffeomorphism $\phi$ of $M$ gives rise to a map $\tilde \phi$ on the total space of this bundle. In particular, any vector field $V$ on $M$ induces a vector field $\tilde V$ on $\mathcal B$, as seen in the following picture.

\begin{wrapfigure}{l}{0.3\textwidth}
    \centering
    \vspace{-.5cm}
    \includegraphics[width=0.3\textwidth]{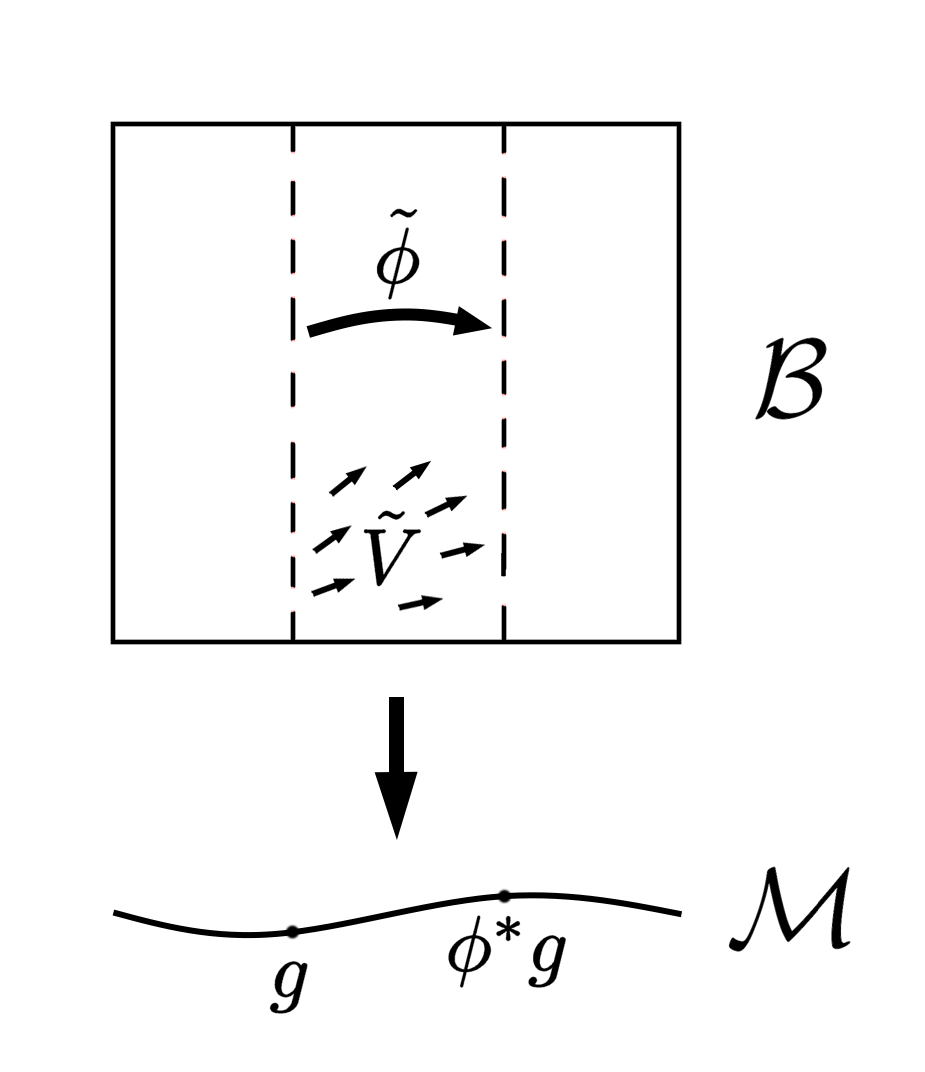}
    \vspace{-.9cm}
\end{wrapfigure}
  One may however wish to define a Lie derivative of pinors as an operator which, when acting on pinors for a metric $g$, again returns pinors for the same metric. In terms of the picture, this corresponds to constructing a \emph{vertical} vector field, i.e.\ one parallel to the fibres. This can be achieved naturally by means of a connection on the bundle $\mathcal B\to\mathcal M$, namely by taking the vertical part of the vector field $\tilde V$. In the paper \cite{bg} the authors clarified how the candidate Lie derivative operator introduced earlier in \cite{k} can be understood by means of such a natural connection on the space of metrics. This topic has also received more attention recently in the mathematical physics literature, cf.\ the recent work \cite{g}.
  
  \subsection{Generalised geometry}
  The main goal of the present work is to revisit this construction and extend it to the context of generalised geometry. The latter is a geometric framework which is naturally linked with geometry of string theory. In particular it has been employed with great success in studying diverse aspects of supergravity theories in 10 dimensions (cf.\ \cite{siegel,csw}). The central object of generalised geometry is a \emph{Courant algebroid}, whose sections can in some sense --- in particular in the context of supergravity --- be regarded as analogues of vector fields on a manifold. One basic example, relevant for $N=1$ supergravity is given by the direct sum
  \[TM\oplus T^*M\oplus \smash{\on{ad}_P},\]
  where $\on{ad}_P$ is the adjoint bundle associated to a principal $G$-bundle with vanishing first Pontryagin class.
  
  As part of its defining data, any Courant algebroid carries a bracket, generalising the commutator of vector fields. This in turn naturally leads to the notion of a \emph{generalised Lie derivative}, in which a section of the Courant algebroid acts on geometric objects associated with the algebroid. One such object is the \emph{generalised metric}, which in the case of $N=1$ supergravity encodes the bosonic fields of the theory.
  
  Assuming a mild topological condition one can lift a generalised metric to a spin structure.
  Sections of the associated spinor bundles then describe the remaining fermionic fields of the theory. Understanding the symmetry structure of supergravity in the generalised framework thus naturally leads to the introduction of the generalised Lie derivative of the spinor fields. This has been successfully employed in creating a Batalin--Vilkovisky formulation of the $N=1$ supergravity in 10 dimensions \cite{ksv}.\footnote{A related construction, describing the generalised Lie derivative of the generalised vielbeine, has also recently appeared in the context of double field theory \cite{blr}.} In this paper we clarify the geometric meaning of the generalised Lie derivative of spinors and in particular explain its relation to generalised Levi-Civita connections via the formula (see Theorem \ref{ggspin} and \eqref{spinor}):
  \[\mathcal L_u \psi= D_u\psi+\tfrac12(D_a u_b)\gamma^{ab}\psi.\]

  \subsection{Outline}
    The paper is structured as follows. We start with recalling the classical setup in Section \ref{sec:ord}: after discussing the relevant connection on the space of metrics, we develop the concept of the metric Lie derivative and explain how it leads to the natural notion of Lie derivative of pinors and spinors. Section \ref{sec:gg} is devoted to the basics of generalised geometry, introducing Courant algebroids and geometric structures on them. Finally, in Section \ref{sec:gglie} we study the analogue of the metric Lie derivative in generalised geometry and prove the basic structural results, including the relation of the (s)pinor Lie derivative and the Levi-Civita connection.
    
  \subsection*{Acknowledgments}
  F.V.\ and F.\v D.\ are supported by the grant PRIMUS/25/SCI/018 of Charles University. F.\v D. is also supported by the grant GA\v CR 26-23375S.
  
\section{Metric Lie derivative}\label{sec:ord}
In this section we introduce the \emph{metric Lie derivative} as a natural vector field on the principal bundle of orthonormal frames for a given metric. This will act on general tensor fields (although in a way different from the usual Lie derivative); lifting this vector field to a principal $\pin(r,s)$-bundle we then obtain the sought-for Lie derivative of pinor fields.

\subsection{Connection on the space of metrics}\label{conmet}
Let us begin with introducing an important piece of the puzzle which is later used in defining the metric Lie derivative. Namely, following \cite{bg}, we will show that there exists a canonical connection on a principal bundle over the space of metrics on a given vector space \( V \).

\begin{defn}
    Fix two non-negative integers $r,s$ and let $V$ be a vector space with $\dim V=r+s$. Denote by $\on{Met}(V)$ the space of inner products (metrics) on $V$ of signature $(r,s)$. We then define the \emph{tautological vector bundle} \[\on{Tau}(V)\to \on{Met}(V)\] to be the vector bundle with inner product, whose fibre over $\gamma\in\on{Met}(V)$ is given by $(V,\gamma)$. Similarly, we define
    \[\on{Bas}(V)\to \on{Met}(V)\]
    to be the principal $O(r,s)$-bundle of orthonormal bases of $\on{Tau}(V)$.\footnote{By an orthonormal basis for an inner product of signature $(r,s)$ we mean a basis in which the inner product takes the form $\on{diag}(\underbrace{1,\dots,1}_r,\underbrace{-1,\dots,-1}_s)$.} Thus the fibre \( \on{Bas}_{\gamma}(V) \) over the point \( \gamma \in \on{Met}(V) \) consists of \( \gamma \)-orthonormal bases of \( V \). 
\end{defn}
Let $\gamma(t)$, $t\in[0,1]$ be a curve in $\on{Met}(V)$, i.e.\ a family of inner products on $V$. Define the associated family of endomorphisms $\tau(t)$ of $V$ by the conditions
\[\tau(0)=\on{id}_V,\qquad \dot\tau=-\tfrac12 \gamma^{-1}\dot \gamma\tau,\]
where in the last equality we interpret both $\dot\gamma$ and $\gamma^{-1}$ as linear maps, i.e.\ in components we have $\dot\tau^\mu{}_\sigma=-\tfrac12\gamma^{\mu\nu}\dot\gamma_{\nu\rho}\tau^\rho{}_\sigma$. 
\begin{theorem}
    The map $\tau(t)$ is an isometry between $(V,\gamma(0))$ and $(V,\gamma(t))$. Thus $\tau$ can be interpreted as the parallel transport w.r.t.\ a natural connection $\nabla^{\mathrm{tau}}$ on the tautological vector bundle, which preserves the inner product. Equivalently, it defines a connection on $\nabla^{\mathrm{bas}}$ on $\on{Bas}(V)$.
\end{theorem}
\begin{proof}
    It suffices to check that for any $v\in V$ the pairing $\gamma(\tau(v),\tau(v))$ is constant in $t$:
    \[\tfrac{d}{dt}\gamma(\tau v,\tau v)=\dot\gamma(\tau v,\tau v)+2\gamma(\dot\tau v,\tau v)=\dot\gamma(\tau v,\tau v)-\gamma(\gamma^{-1}\dot \gamma \tau v,\tau v)=\dot\gamma(\tau v,\tau v)-\dot\gamma(\tau v,\tau v)=0.\qedhere\]
\end{proof}
\begin{remark}\label{rk:idea-classical}
  The above can be summarised by saying that \[\on{id}-\tfrac12\epsilon\gamma^{-1}\delta\gamma\] is, up to order $\epsilon$, a natural isometry between $(V,\gamma)$ and $(V,\gamma+\epsilon\delta\gamma)$. In particular, any orthonormal basis $\{e_i\}$ w.r.t.\ $\gamma$ defines a new orthonormal basis $\{e_i-\tfrac12\epsilon \gamma^{-1}\delta\gamma(e_i)\}$ for $\gamma+\epsilon\delta\gamma$.
\end{remark}

\subsection{Intermezzo: Lie derivative as a vector field lift}
Before proceeding, let us recall the following fact. For any principal $G$-bundle $P\to M$ and a representation $\rho\colon G\to GL(V)$, there is a natural identification between sections of the associated bundle $P\times_\rho V$ and $G$-equivariant smooth maps $P\to V$. Let us denote the latter by
\[C^\infty(P,V)^G=\{T\colon P\to V\mid R_g^*T=\rho(g^{-1})T\quad\forall g\in G\},\]
where $R_g$ is the right $G$-action on $P$.
For instance, if we take $P$ to be the frame bundle $\on{Fr}(M)$ of a smooth $n$-dimensional manifold $M$, with $G=GL(n,\mathbb R)$ and $\rho$ a tensor representation, then we obtain the interpretation of tensor fields on $M$ in terms of equivariant functions on $\on{Fr}(M)$.

Importantly, the Lie derivative $\mathcal L_X$ w.r.t.\ a vector field $X\in\mathfrak X(M)$ can be regarded as a $GL(n,\mathbb R)$-invariant vector field $X^\text{Fr}$ on $\on{Fr}(M)$, lifting the vector field $X$ on $M$. More concretely, at any frame $f\in\on{Fr}(M)$ we have
\begin{equation}\label{xfr}
X^{\text{Fr}}_f := \frac{ \text{d} }{ \text{d} t } \bigg{\vert}_{t=0} \phi_{t*} f\;\in T_f\on{Fr}(M),
\end{equation}
where $\phi_t$ is the flow of $X$ and we see $\phi_{t*}$ as acting on all the elements of $f$. Regarding any tensor field (or more generally a tensor density) as an equivariant function $T$ on $\on{Fr}(M)$, we then have
\[\mathcal{L}_X T = X^{\text{Fr}  } \, T.\]

We see that the Lie derivative can be understood in multiple ways: as an operator on the space of vector fields, as a vector field on the principal bundle of frames, and finally as a vector field on the total space of the tangent bundle (corresponding to the infinitesimal pushforward of vectors). This is a more general phenomenon --- for any vector field $X\in\mathfrak X(M)$ there exists an equivalence between the following objects:
\begin{itemize}
  \item $GL(n,\mathbb R)$-invariant vector fields on $\on{Fr}(M)$ covering $X$
  \item vector fields on the total space of $TM$ whose flow preserves the linear structure on $TM$ and which cover $X$
  \item $\mathbb R$-linear operators $\varpi\colon\mathfrak X(M)\to \mathfrak X(M)$ s.t.\ for every $Y\in\mathfrak X(M)$, $f\in C^\infty(M)$ we have
  \[\varpi(fY)=f\varpi(Y)+(Xf)Y.\]
\end{itemize}
Equivalently, such objects can be described as sections of the Atiyah Lie algebroid $\on{At}(\on{Fr}(M))$ which map to $X$ under the associated map $\on{At}(\on{Fr}(M))\to TM$. The equivalence above also naturally extends to a more general context, e.g.\ to vector bundles carrying inner product.

\subsection{Metric Lie derivative}\label{subsec:metric}
The fundamental problem in defining the Lie derivative of a pinor or a spinor field is that these fields are not associated to a $GL(n,\mathbb R)$-representation and hence they cannot be regarded as equivariant functions on the bundle $\on{Fr}(M)$ --- instead, they are equivariant functions on the corresponding principal $Pin(r,s)$ or $Spin(r,s)$ bundle. Our goal is therefore to construct a natural lift of any vector field $X$ to an invariant vector field on these bundles. We start by constructing a lift to the principal bundle of orthonormal frames on a pseudo-Riemannian manifold.

The problem in repeating the definition \eqref{xfr} verbatim is that if $f$ is an orthonomal frame for the metric $g$, its pushfoward $\phi_{t*}f$ will be orthonormal w.r.t.\ the (in general different) metric $\phi_{t*}g$. We can nevertheless produce a $g$-orthonomal frame out of $\phi_{t*}f$ by parallel transporting it via the connection $\nabla^{\text{bas}}$ along the curve of metrics given by $\phi_{s*}g$. This leads to the following definition.

\begin{defn}
  Let $(M,g)$ be a pseudo-Riemannian manifold with a metric of signature $(r,s)$ and $X\in\mathfrak X(M)$ a vector field. Denote by $O_gM$ the associated principal $O(r,s)$-bundle of orthonormal frames and write $\phi_t$ for the flow of $X$. We define the lift $X^O\in\mathfrak X(O_gM)$ of $X$ by
  \[X^O_f := \frac{ \text{d} }{ \text{d} t } \bigg{\vert}_{t=0} \tau_t (\phi_{t*} f),\qquad f\in (O_gM)_m,\quad m\in M,\]
  where for any fixed $t$, the map $\tau_t$ is the parallel transport in the bundle $\on{Bas}(T_{\phi_t(m)}M)\to \on{Met}(T_{\phi_t(m)}M)$ along the curve \[\gamma(s)=\phi_{(t-s)*}(g_{\phi_{s}(m)})\in \on{Met}(T_{\phi_t(m)}M),\quad s\in[0,t].\] 
  By construction, $X^O$ is $O(r,s)$-invariant.
  The \emph{metric Lie derivative} $\mathcal L_X^g$ of any tensor field $T$ (or tensor density) on $M$, regarded as an equivariant function on $O_gM$, is then defined as
  \[\mathcal{L}^g_X T := X^O T.\]
\end{defn}
\noindent\begin{minipage}{0.65\textwidth}
\begin{remark}
  Although the above description of the parallel transport is valid for finite values of $t$, the definition of the metric Lie derivative requires only the knowledge of the parallel transport to the first order in $t$. The construction of the vector field $X^O$ can thus be understood in more simple terms as follows: We first flow the frame $f$ at $m$ infinitesimally along the vector field $X$, ending with an orthonormal frame for the metric $\phi_{\epsilon*}g$ at $\phi_\epsilon(m)$. Since for each $m \in M $ and each metric $ g $, the fibre $(O_gM)_m $ coincides with $ \on{Bas}_g(T_mM)$, the connection $\nabla^{\text{bas}}$ on the space of metrics on $T_{\phi_\epsilon(m)}M$ may be used to identify the space of orthonormal frames for the infinitesimally near pair of metrics $\phi_{\epsilon*}g$ and $g$, and produce an orthonormal frame for $g$ at $\phi_\epsilon(m)$. This is captured in the following picture.
\end{remark}
\end{minipage}
\begin{minipage}{0.34\textwidth}
\vspace{-.3cm}
\includegraphics[width=\textwidth]{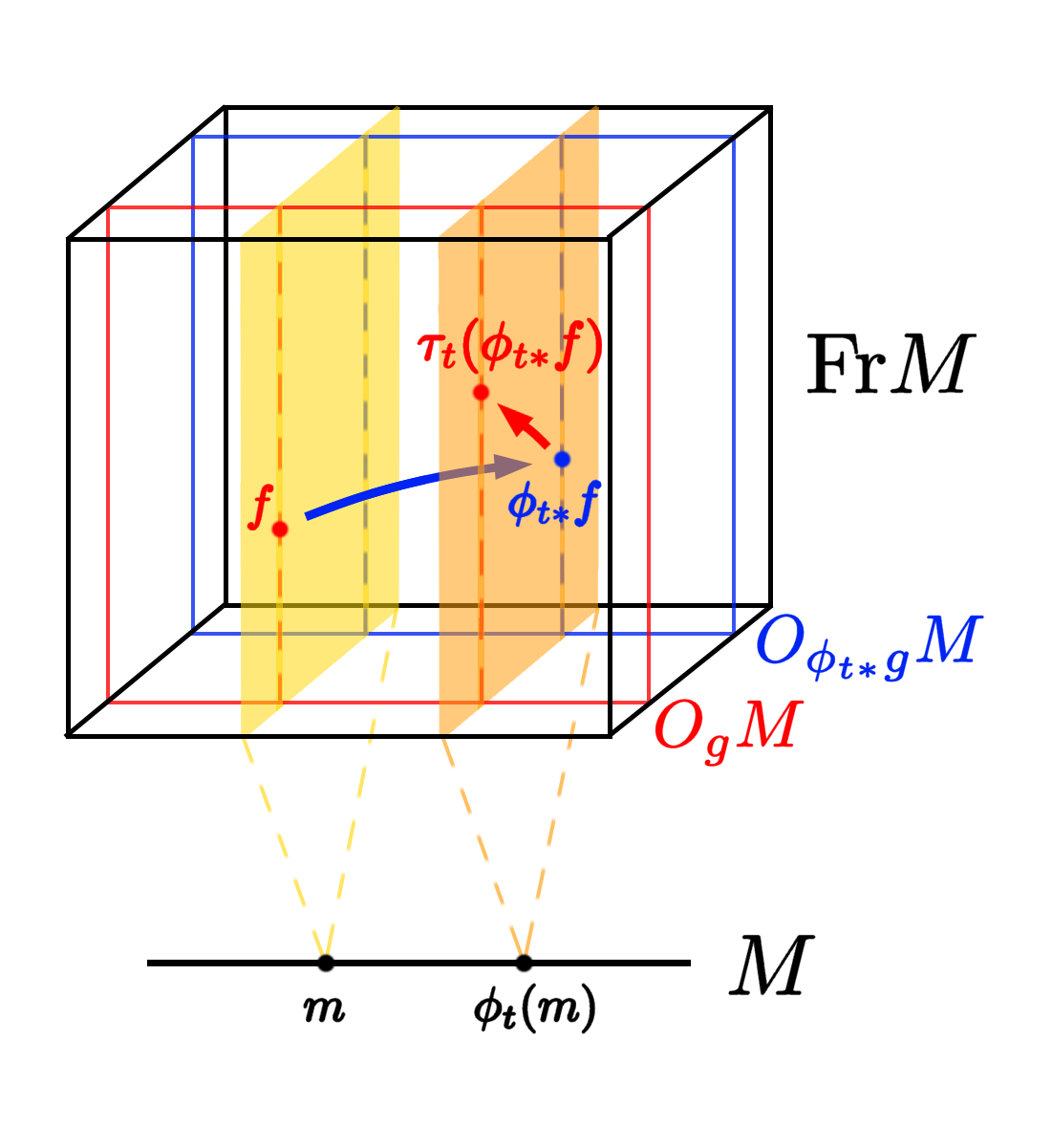}
\vspace{-.9cm}
\end{minipage}
\begin{remark}\label{rem:natural}
  Note that the connection on $\on{Bas}(V)\to\on{Met}(V)$ is natural in the following sense: if $\omega\colon V\to W$ is an isomorphism of vector spaces and $\tau^\gamma_t$ is the parallel transport along a curve $\gamma$ in $\on{Met}(W)$ then
  $\smash{\tau^\gamma_t\omega=\omega\,\tau_t^{\omega^*\gamma}}$.
  Taking $\omega=\phi_{t*}\colon T_mM\to T_{\phi_t(m)}M$, this leads to an equivalent definition of $X^O$:
  \[X^O_f := \frac{ \text{d} }{ \text{d} t } \bigg{\vert}_{t=0} \phi_{t*} \tau_t (f),\]
  where now $\tau_t$ is the parallel transport in the bundle $\on{Bas}(T_mM)\to\on{Met}(T_mM)$ along the curve \[\gamma(s)=\phi_s^*g_{\phi_s(m)}=(\phi_s^*g)_m\in\on{Met}(T_mM),\quad s\in [0,t].\]
\end{remark}

As mentioned before, the action of the metric Lie derivative in general differs from the action of the standard one, even on simple objects such as vector fields on $M$. The following statement makes this precise.
\begin{theorem}\label{metlierezy}
Let $(M,g)$ be a pseudo-Riemannian manifold and $X,V\in \mathfrak X(M)$. Then
\[\mathcal{L}^g_X V = \mathcal{L}_X V + \tfrac{1}{2} g^{-1} (\mathcal{L}_Xg) V, \]
where we regard $\mathcal L_Xg$ and $g^{-1}$ as maps $TM\to T^*M$ and $T^*M\to TM$, respectively.
\end{theorem}
\begin{proof}
Let \( \tilde{V} \in C^\infty (O_gM , \mathbb{R}^{r+s} )^{O(r,s)}\) denote the function representing a vector field \( V \in \mathfrak X(M) \) (here $\mathbb R^{r+s}$ is the standard vector representation of $O(r,s)$) and let \( \bar\tau_t \in \on{End} (T_m M) \) denote the parallel transport on \( \on{Tau}(T_mM) \to \on{Met}(T_mM)\) along the curve \( \gamma(s) := (\phi^*_sg)_m \), $s\in[0,t]$. One has
\[ \bar\tau_t := \on{id} - \tfrac{1}{2} t g^{-1} (\mathcal{L}_Xg) + \mathcal{O}(t^2).  \]
As before, \( \tau_t\colon \on{Bas}_g(T_mM) \to \on{Bas}_{\phi^*_tg}(T_mM) \) denotes the corresponding transport on the principal bundle.  At a point \( m \in M \) for a frame \( f \equiv \{ e_i \} \in (O_gM)_m \) we have $\tau_tf=\{\bar\tau_t e_i\}$ and so (for clarity we will here write the dependence on points in parentheses)
\begin{align*}
[ X^O \tilde{V}(f) ]^i e_i &= \left[ \frac{ \text{d} }{ \text{d} t} \bigg{\vert}_{t=0} \tilde{V}^i (\phi_{t*} \circ \tau_t f ) \right] e_i= \frac{ \text{d} }{ \text{d} t} \bigg{\vert}_{t=0}\tilde{V}^i (\phi_{t*} \circ \tau_t f ) \bar\tau_{-t} \circ \phi_{(-t)*} \circ \phi_{t*} \circ \bar\tau_t e_i \\
&=  \frac{ \text{d} }{ \text{d} t} \bigg{\vert}_{t=0} \bar\tau_{-t} \circ \phi_{(-t)*} V(\phi_t m) = \frac{ \text{d} }{ \text{d} t} \bigg{\vert}_{t=0}\phi_{(-t)*} V(\phi_t m) + \left( \frac{ \text{d} }{ \text{d} t} \bigg{\vert}_{t=0} \bar\tau_{-t} \right) V(m)\\
&= \mathcal{L}_XV(m) + \frac{1}{2} g^{-1} (\mathcal{L}_Xg) V(m).\qedhere
\end{align*}
\end{proof}
\begin{remark}
  Since $\mathcal L_X^g$ is defined as the action of an $O(r,s)$-invariant vector field, it automatically satisfies the Leibniz rule w.r.t.\ the tensor product, preserves the metric $g$, and commutes with all contractions. For instance, for any $\alpha \in\Omega^1(M)$ and $V\in\mathfrak X(M)$ we have
  \[\left< \mathcal{L}^g_X \alpha , V \right> =  \mathcal{L}^g_X  \langle \alpha , V \rangle-\left< \alpha , \mathcal{L}^g_X V \right>=X  \langle \alpha , V \rangle-\left< \alpha , \mathcal{L}^g_X V \right>.\]
\end{remark}
As an important corollary of the above Theorem we obtain the relation between the metric Lie derivative and the Levi-Civita connection $\nabla$ for $g$.
\begin{cor}\label{metliekovvekt}
For $X,Y\in\mathfrak X(M)$ and for any local frame $\{E_\mu\}$ on $M$ we have
\[ (\mathcal{L}^g_X Y)^\mu = X^\nu\nabla_\nu Y^\mu - Y_\nu\nabla^{[\nu} X^{\mu]}.\]
\end{cor}
\begin{proof}
Using Theorem \ref{metlierezy} and $(\mathcal{L}_Xg)(Y,Z) = g(\nabla_Y X ,Z) + g(Y ,\nabla_Z X)$ we obtain for any $Z\in\mathfrak X(M)$:
\begin{align*}
  g(\mathcal L_X^gY,Z)&=g([X,Y],Z)+\tfrac12(\mathcal L_Xg)(Y,Z)=g(\nabla_XY-\nabla_YX,Z)+\tfrac12g(\nabla_Y X ,Z) + \tfrac12g(Y ,\nabla_Z X)\\
  &=g(\nabla_XY-\tfrac12\nabla_YX+\tfrac12g^{-1}(g(Y,\nabla X)) ,Z).\qedhere
\end{align*}
\end{proof}
Note that the Levi-Civita connection corresponds precisely to the horizontal lift $X^{LC}$ of the vector field $X$ to $O_gM$, i.e.\ $X^{LC}T=\nabla_X T$ for any equivariant function $T$ on $O_gM$. Thus the Corollary can be restated as follows:
\begin{cor}\label{rem:difference}
For any $X\in\mathfrak X(M)$ we have
  \begin{equation}\label{metric-vs-lc}
    X^O=X^{LC}+\on{antisym}(\nabla X),
  \end{equation}
  where the last term is understood in terms of the (infinitesimal) action of $O(r,s)$ on $O_gM$.\footnote{We follow here the convention for the coefficients of $A$ where $(Av)^\mu=A^{\mu}{}_\nu v^\nu$.}
\end{cor}

\subsection{Lie derivative of pinor and spinor fields}\label{secliepin}
Let us start by briefly recalling some basic facts about the Lie group $\pin(r,s)$. Recall that this is the double cover of $O(r,s)$, via a homomorphism
\[\varphi\colon \pin(r,s)\to O(r,s),\]
with $\ker\varphi\cong \mathbb Z_2$. This in particular induces an isomorphism of the Lie algebras $\mathfrak{pin}(r,s)\cong \mathfrak o(r,s)$. The group $\pin(r,s)$ admits an important (complex) representation $S$, called the \emph{pin representation}. Under the identification $\mathfrak{pin}(r,s)\cong \mathfrak o(r,s)$ this is given by
\[\mathfrak o(r,s)\ni A\mapsto \tfrac14A_{\mu\nu}\gamma^{\mu\nu}\in\on{End}(S),\]
where $\gamma^{\mu\nu}=\gamma^{[\mu}\gamma^{\nu]}$ and the \emph{gamma matrices} $\gamma^\mu\in\on{End}(S)$ satisfy the relation $\gamma^\mu\gamma^\nu+\gamma^\nu\gamma^\mu=2g^{\mu\nu}$.

Recall also that the \emph{pin structure} on a manifold $M$ is the choice of a pseudo-Riemannian metric $g$ together with a lift of the $O(r,s)$-structure to a $\pin(r,s)$-one, i.e.\ a principal $\pin(r,s)$-bundle $P_gM\to M$, together with a double cover $\pi\colon P_gM\to O_gM$ compatible with the respective right actions in the fibres:
\[\pi \circ R^{\text{Pin}(r,s)}_A = R^{O(r,s)}_{\varphi(A)} \circ \pi   \;\;\;\;\;\;\; \forall \, A \in \text{Pin}(r,s).\] 

Since $\pi$ is a local diffeomorphism, for every $p\in P_gM$ the induced map $\pi_*\colon T_p(P_gM)\to T_{\pi(p)}(O_gM)$ is an isomorphism. Thus, any vector field on $O_gM$ automatically lifts to a vector field on $P_gM$.
\begin{lemma}\label{lemma:lift}
  The lift of any $O(r,s)$-invariant vector field on $O_gM$ to $P_gM$ is $\pin(r,s)$-invariant.
\end{lemma}
\begin{proof}
  Starting with an $O(r,s)$-invariant vector field $V$, let us call its lift $V'$. Then for any $p\in P_gM$ and $A\in\pin(r,s)$ we have
  \[\pi_*R^{\pin(r,s)}_{A*}V'_p=R_{A*}^{O(r,s)}\pi_*V'_p=R_{A*}^{O(r,s)}V_{\pi(p)}=V_{R_{A}^{O(r,s)}\pi(p)}=V_{\pi R_{A}^{\pin(r,s)}p}=\pi_*V'_{R_{A}^{\pin(r,s)}p}.\]
  The result follows from the fact that for any $p$ the map $\pi_*$ is an isomorphism.
\end{proof}
\begin{remark}
  Note that the proof of the Lemma does not depend of the particularities of this case and only uses the fact that $\pin(r,s)\to O(r,s)$ is a local diffeomorphism. Furthermore, since the map $P_gM\to O_gM$ is compatible with the group actions and $\mathfrak{pin}(r,s)\cong \mathfrak o(r,s)$, it follows that the fundamental vector fields on $O_gM$ lift precisely to the corresponding fundamental vector fields on $P_gM$.
\end{remark}
\begin{example}
  The lift of the vector field $X^{LC}$ defines the lift of the Levi-Civita connection to $P_gM$. This is often called the \emph{pin connection} and is denoted simply by $\nabla$.
\end{example}
\begin{defn}
  Let $M$ be a manifold with a pin structure. For any vector field $X\in\mathfrak X(M)$ we take $X^P\in\mathfrak X(P_gM)$ to be the lift of $X^O\in\mathfrak X(O_gM)$. Let $\Sigma$ be a representation of $\pin(r,s)$ and $\Xi$ a section of the associated bundle, or equivalently $\Xi\in C^\infty(P_gM,\Sigma)^{\pin(r,s)}$. We then define its \emph{Lie derivative} by
  \[\mathcal L_X\Xi:=X^P\Xi.\]
\end{defn}
\begin{remark}
  Note that in this case we dropped the superscript $g$ and the adjective ``metric'', due to the fact that there is no direct analogue of the ordinary ``non-metric'' Lie derivative in this case.
\end{remark}

As an immediate corollary of \eqref{metric-vs-lc} we then obtain:
\begin{theorem}
  Let $M$ be a manifold with a pin structure and $\psi$ a pinor field, i.e.\ a section of the vector bundle associated to the pinor representation. Then for any $X\in\mathfrak X(M)$ we have
  \[\mathcal L_X\psi=\nabla_X\psi+\tfrac14(\nabla_\mu X_\nu)\gamma^{\mu\nu}\psi.\]
\end{theorem}

\begin{remark}
  The last part of the above constructions can be succinctly summarised in terms of Atiyah algebroids: the metric Lie derivative $\mathcal L_X^g$ as well as the Levi-Civita connection $\nabla_X$ can be both regarded as sections of $\on{At}(O_gM)$. Since they map to the same vector field on $M$, their difference is a section of $\Lambda^2T^*M$ --- by Corollary \ref{rem:difference} this is precisely $\on{antisym}(\nabla X)$. The double cover $\pi\colon P_gM\to O_gM$ induces an isomorphism $\on{At}(O_gM)\cong \on{At}(P_gM)$ --- the image of the metric Lie derivative and of the Levi-Civita connection are precisely the Lie derivative (of pinors, etc.) and the pin connection, respectively.
\end{remark}

To finish, let us connect the story to the (perhaps more well-known) case of spinors. Recall that $\on{Spin}(r,s)$ is the double cover of the group $SO(r,s)$, and can be seen as $\varphi^{-1}(SO(r,s))\subset \pin(r,s)$. Analogously to the case above, a \emph{spin structure} on a manifold $M$ consists of a choice of an orientation and a pseudo-Riemannian metric $g$ (producing an $SO(r,s)$-principal bundle $SO_gM$ of oriented orthonormal frames on $M$) together with a lift of the $SO(r,s)$-structure to a $\on{Spin}(r,s)$-one. Note that we have an analogue of Lemma \ref{lemma:lift} to the case of $SO$ versus $\on{Spin}$, and that $SO_gM\subset O_gM$ is open.
\begin{defn}\label{defn:spin}
  Let $M$ be a manifold with a spin structure. For any vector field $X\in\mathfrak X(M)$ we take $X^{Sp}$ to be the lift of $X^O|_{SO_gM}\in\mathfrak X(SO_gM)$ to the principal $\on{Spin}(r,s)$-bundle. Let $\Sigma$ be a representation of $\on{Spin}(r,s)$ and $\Xi$ a section of the associated bundle, or equivalently an equivariant function on the principal $\on{Spin}(r,s)$-bundle. We then define the \emph{Lie derivative} by
  \[\mathcal L_X\Xi:=X^{Sp}\Xi.\]
\end{defn}
In particular, if $\Sigma$ is the spinor representation (i.e.\ the representation induced by the pinor representation of $\pin(r,s)$) then for any corresponding section $\psi$ we again get
\[\mathcal L_X\psi=\nabla_X\psi+\tfrac14(\nabla_\mu X_\nu)\gamma^{\mu\nu}\psi.\]

\begin{remark}
  Lastly, we remark that there is in fact a simpler geometric setup (apart from the (s)pin case), where one is forced to introduce the metric Lie derivative instead of the ordinary one --- this is the case of (anti-)self-dual differential forms on oriented pseudo-Riemannian manifolds.
\end{remark}

\section{Generalised geometry}\label{sec:gg}

\subsection{Courant algebroids}
Let us now turn to the generalised-geometric counterpart of the story.
We start by recalling the definition of the central structure.
\begin{defn}[\cite{lwx,let}]
  A \emph{Courant algebroid} is a vector bundle $E\to M$ together with
  \begin{itemize}
    \item an $\mathbb R$-bilinear operation $[\cdot,\cdot]\colon \Gamma(E)\times\Gamma(E)\to\Gamma(E)$
    \item a fibrewise inner product $\langle \cdot,\cdot\rangle$ of signature $(p,q)$
    \item a vector bundle map $\rho\colon E\to TM$,
  \end{itemize}
  such that for all $u,v,w\in\Gamma(E)$ and $f\in C^\infty(M)$ we have
  \begin{itemize}
    \item[$\circ$] $[u,[v,w]]=[[u,v],w]+[v,[u,w]]$ (the Jacobi identity)
    \item[$\circ$] $[u,fv]=f[u,v]+(\rho(u)f)v$
    \item[$\circ$] $\rho(u)\langle v,w\rangle=\langle [u,v],w\rangle+\langle v,[u,w]\rangle$ (invariance of the pairing) 
    \item[$\circ$] $[u,v]+[v,u]=\mathcal D\langle u,v\rangle$,
  \end{itemize}
  where $\mathcal D:=\rho^*\circ d\colon C^\infty(M)\to \Gamma(T^*M)\to \Gamma(E^*)\cong \Gamma(E)$, where in the last step we have identified $E$ with $E^*$ using the pairing.
\end{defn}
\begin{defn}
Note that the middle two axioms together imply that the bracket with any section $u\in\Gamma(E)$ induces an $O(p,q)$-invariant vector field \[u^{\text{Fr}}\in\mathfrak X(\on{Fr}(E))\] on the principal bundle $\on{Fr}(E)\to M$ of orthonormal frames of $E$, covering the vector field $\rho(u)$ on $M$. 
For any equivariant function $t$ on $\on{Fr}(E)$ we then define the \emph{generalised Lie derivative} by
\[\mathcal L_ut:=u^{\text{Fr}}t.\]
\end{defn}
\begin{remark}\label{rem:flows}
  Let $\Phi_t$ be the flow of the invariant vector field $u^{\text{Fr}}$, covering the flow $\bar\Phi_t$ of $\rho(u)\in\mathfrak X(M)$. We will also use the same notation $\Phi_t$ for the associated inner-product-preserving flow on the total space of $E\to M$. Then the above implies that for any $v\in\Gamma(E)$
  \begin{equation}\label{auto}
    \frac{ \text{d} }{ \text{d} t } \bigg{\vert}_{t=0} \Phi_{-t}\circ v \circ \bar\Phi_t=[u,v]\in\Gamma(E).
  \end{equation}
\end{remark}
A fairly large class of examples of Courant algebroids, which are particularly relevant in the context of $N=1$ supergravity in 10 dimensions, is given as follows. 
\begin{example}\label{main-ex}
  Let $G$ be a Lie group with an invariant inner product on its Lie algebra $\mathfrak g$. Let furthermore $P\to M$ be a principal $G$-bundle with a connection $\nabla^G$ of curvature $F$, and let $H$ be a 3-form on $M$ such that\footnote{Note that this condition forces $P$ to have vanishing first Pontryagin class.}
  \[dH=\langle F,F\rangle_\mathfrak g.\]
  Denoting the bundle associated to the adjoint representation by $\text{ad}_P$, we now define a Courant algebroid:
  \[E=TM\oplus T^*M\oplus \text{ad}_P,\qquad \langle X+\alpha+s,Y+\beta+t\rangle=\alpha(Y)+\beta(X)+\langle s,t\rangle_\mathfrak g,\qquad \rho(X+\alpha+s)=X,\]
  \begin{align*}
    [X+\alpha+s,Y+\beta+t]&=\mathcal L_XY+(\mathcal L_X\beta-i_Yd\alpha+i_Yi_XH+\langle \nabla^G\!s,t\rangle_\mathfrak g-\langle i_XF,t\rangle_\mathfrak g+\langle i_YF,s\rangle_\mathfrak g)\\
    &\qquad+(\nabla_X^Gt-\nabla_Y^Gs+[s,t]_\mathfrak g+i_Yi_XF).
  \end{align*}
\end{example}

\subsection{Generalised metrics and Levi-Civita connections}
  \begin{defn}
    A \emph{generalised metric} on a Courant algebroid $E\subset M$ is a subbundle $V_+\subset E$ such that the induced inner product $\langle \cdot,\cdot\rangle|_{V_+}$ is non-degenerate. We set $V_-:=V_+^\perp$ and we will write $(\cdot)_\pm$ for the orthogonal projection onto $V_\pm$.
  \end{defn}
  \begin{lemma}
    Generalised metric is equivalent to a symmetric endomorphism $\mathcal G\colon E\to E$ s.t.\ $\mathcal G^2=\on{id}$.
  \end{lemma}
  \begin{proof}
    Given $\mathcal G$ we take $V_+$ to be its $+1$-eigenbundle. Starting with a $V_+$ we construct $\mathcal G$ as the reflection w.r.t.\ $V_+$.
  \end{proof}
  \begin{lemma}\label{lem:killing}
    For any $u,v\in\Gamma(E)$ we have $(\mathcal L_u\gm)v_\pm=\pm 2[u,v_\pm]_\mp$.
  \end{lemma}
  \begin{proof}
    $(\mathcal L_u\gm)v_\pm=\mathcal L_u(\gm v_\pm)-\gm\mathcal L_uv_\pm=(\pm 1-\gm)\mathcal L_u v_\pm=\pm2[u,v_\pm]_\mp$.
  \end{proof}
  \begin{defn}
    Let $E\to M$ be a Courant algebroid of signature $(p,q)$ and $V_+$ a generalised metric of signature $(r,s)$, i.e.\ $V_+$ and $V_-$ have signatures $(r,s)$ and $(p-r,q-s)$, respectively. We define \[O_\gm M:=\{\text{orthonormal bases of $E$, adapted to the splitting $E=V_+\oplus V_-$}\}.\]
    This is a principal $O(r,s)\times O(p-r,q-s)$-subbundle of the principal $O(p,q)$-bundle $\on{Fr}(E)$ over $M$.
  \end{defn}
  \begin{example}
    Continuing with Example \ref{main-ex}, for any pseudo-Riemannian metric $g$ on $M$ we can take \[V_+=\on{graph}(g)=\{x+i_xg\mid x\in TM\}.\] Thus, any triple $(g,H,\nabla^G)$ of a pseudo-Riemannian metric $g$, 3-form $H$, and a connection $\nabla^G$ on $P$ satisfying $dH=\langle F,F\rangle_\mathfrak g$ can be used to build a pair of a Courant algebroid and a generalised metric. This is how the bosonic field content of $N=1$ supergravity is encoded in generalised geometry.
  \end{example}
  \begin{defn}
    A \emph{generalised Levi-Civita connection} is an $\mathbb R$-bilinear map
    \[D\colon \Gamma(E)\times\Gamma(E)\to\Gamma(E),\qquad (u,v)\mapsto D_uv,\]
    such that for all $u,v,w\in\Gamma(E)$ and $f\in C^\infty(M)$ we have
    \begin{itemize}
      \item[$\circ$] $D_{fu}v=fD_uv$ (tensoriality in the first slot)
      \item[$\circ$] $D_u(fv)=fD_uv+(\rho(u)f)v$ (derivative property in the second slot)
      \item[$\circ$] $\rho(u)\langle v,w\rangle=\langle D_uv,w\rangle+\langle v,D_uw\rangle$ (preservation of the pairing)
      \item[$\circ$] $D_u\Gamma(V_+)\subset \Gamma(V_+)$ (preservation of the generalised metric)
      \item[$\circ$] $D_uv-D_vu-[u,v]+\langle Du,v\rangle=0$ (vanishing of the generalised torsion).
    \end{itemize}
  \end{defn}
  The basic structural result is:
  \begin{theorem}[\cite{gf}]
    Generalised Levi-Civita connections exist (but are typically not unique).
  \end{theorem}
  Note that the third and fourth axioms imply also $D_u\Gamma(V_-)\subset \Gamma(V_-)$. Thus:
  \begin{defn}
    Given a generalised Levi-Civita connection $D$ and a section $u\in\Gamma(E)$, the operator $D_u$ induces an invariant vector field \[u^{LC}\in\mathfrak X(O_\gm M),\] covering the vector field $\rho(u)$ on $M$. (We employed the notation $u^{LC}$ to parallel the previous more classical exposition, but the reader should keep in mind the fact that $u^{LC}$ in fact depends on the choice of the generalised Levi-Civita connection $D$.)
  \end{defn}
\section{Metric Lie derivative on Courant algebroids}\label{sec:gglie}

\subsection{Connection on the space of generalised metrics}
Following Section \ref{sec:ord}, let us first discuss the linear case, which will later correspond to a fibre of the Courant algebroid.

\begin{defn}
Fix a pair $(r,s)$ of non-negative integers. Consider a vector space $W$ with a non-degenerate pairing $\left< , \right> $ of signature $(p,q)$. 
Let $\on{Met}(W)$ be the space of all linear subspaces $V_+\subset W$ s.t.\ the induced pairing $\langle\cdot,\cdot\rangle|_{V_+}$ has signature $(r,s)$. 
We define the \emph{tautological vector bundles} \[\on{Tau}^\pm(W)\to\on{Met}(W)\] to be the vector bundles with inner product, whose fibres over $V_+\in\on{Met}(W)$ are given by $(V_\pm,\langle \cdot,\cdot\rangle|_{V_\pm})$. Similarly, we define
    \[\on{Bas}^\pm(W)\to \on{Met}(W)\]
    to be the principal $O(r,s)$-bundle and principal $O(p-r,q-s)$-bundle whose fibres are orthonormal bases of $V_+$ and $V_-$, respectively. We then take
    \[\on{Bas}(W):=\on{Bas}^+(W)\times \on{Bas}^-(W)\to \on{Met}(W)\]
    to be the principal $O(r,s)\times O(p-r,q-s)$-bundle of orthonormal bases of $W$ adapted to the splitting $W=V_+\oplus V_-$. Note that this is a subbundle of the orthonormal frame bundle of $\on{Tau}^+(W)\oplus \on{Tau}^-(W)=W\times\on{Met}(W)$.   
\end{defn}
\begin{remark}\label{rk:idea}
    The basic idea is again the fact that the nearby fibres of $\on{Tau}^\pm(W)$ and $\on{Bas}(W)$ can be naturally identified, this time using the orthogonal projection. In other words, for any infinitesimally near pair $V_+,V_+'\in\on{Met}(W)$ the following compositions are isometries:
    \begin{equation}\label{proj}
      \begin{tikzcd}
        V_+' \ar[r, hook] & W \ar[r, two heads] & V_+
    \end{tikzcd}\qquad \begin{tikzcd}
        V_-' \ar[r, hook] & W \ar[r, two heads] & V_-
    \end{tikzcd}
    \end{equation}
      We remark that, rather surprisingly, this identification is in a sense even more natural (or ``obvious'') than its classical counterpart in Remark \ref{rk:idea-classical}.
\end{remark}

A more ``finite'' version is as follows. Let $V_+(t)$, $t\in[0,1]$ be a curve in $\on{Met}(W)$, and let $\mathcal G(t)$ be the associated family of symmetric idempotent endomorphisms. Define the family of endomorphisms $\tau(t)$ of $W$ by the conditions
\[\tau(0)=\on{id}_W,\qquad \dot\tau=\tfrac12\dot {\mathcal G}\gm \tau.\]

\begin{theorem}
    For each $t$ the map $\tau(t)\colon W\to W$ is an isometry, which maps $V_\pm(0)$ to $V_\pm(t)$. Thus $\tau$ can be interpreted as the parallel transport w.r.t.\ a natural connection $\nabla^{\mathrm{tau}}$ on the tautological vector bundles, which preserves their respective inner products. Equivalently, it defines a connection on $\nabla^{\mathrm{bas}}$ on $\on{Bas}(W)$.
\end{theorem}
\begin{proof}
    First, we note that $0=\tfrac{d}{dt}\on{id}_W=\tfrac{d}{dt}\gm^2=\dot \gm\gm+\gm\dot\gm=\dot\gm\gm+(\dot\gm\gm)^T$, implying that $\dot\gm\gm\in\mathfrak o(W)$. Thus for any $v\in W$ we have
    \[\tfrac{d}{dt}\langle \tau v,\tau v\rangle=2\langle \tau v,\dot\tau v\rangle=\langle \tau v,\dot\gm\gm \tau v\rangle=0,\]
    and so $\tau(t)$ is an isometry. Finally, fixing $v\in V_+(0)$ we set
    \[a(t):=(1-\gm(t))\tau(t)v\in W.\]
    Differentiating we get the ODE
    \[\dot a=(-\dot\gm\tau+(1-\gm)\dot\tau)v=(-\dot\gm+\tfrac12(1-\gm)\dot\gm\gm)\tau v=(-\dot\gm+\tfrac12\dot\gm(1+\gm)\gm)\tau v=-\tfrac12\dot\gm(1-\gm)\tau v=-\tfrac12\dot\gm a,\]
    which, together with $a(0)=(1-\gm(0))v=0$ (and using the uniqueness of solutions), implies $a(t)\equiv 0$. Thus $\tau(t)$ maps $V_+(0)$ to $V_+(t)$. Since it is an isometry, it also follows that $\tau(t)V_-(0)=V_-(t)$.
\end{proof}
One easily checks that the infinitesimal parallel transport corresponds indeed to the orthogonal projections \eqref{proj}.

\subsection{Generalised metric Lie derivative}
Following the construction of Subsection \ref{subsec:metric} we wish to construct, for any given $u\in\Gamma(E)$ an invariant vector field on the total space of the bundle $O_\gm M$ of orthonormal frames adapted to the decomposition $E=V_+\oplus V_-$.

Following Remark \ref{rem:flows}, let $\Phi_t$ be the flow of the invariant vector field $u^{\text{Fr}}$, covering the flow $\bar\Phi_t$ of $\rho(u)\in\mathfrak X(M)$. We will again use the same notation $\Phi_t$ for the flow on the total space of $E\to M$.
If now $f$ is an orthonormal frame adapted to $V_\pm$ at $m\in M$, $\Phi_t(f)$ will be an orthonormal frame adapted to $\Phi_{t}(V_\pm)$ at the point $\bar\Phi_t(m)$. We then subsequently employ the parallel transport on $\on{Bas}(E_{\bar{\Phi}_t(m)})$ to obtain an orthonormal frame adapted to $V_\pm$. This leads to the following definition.

\begin{defn}
For $ u \in \Gamma(E) $ we define the vector field $ u^O \in \mathfrak{X}(O_\gm M) $ as follows. For any adapted frame $f\in (O_\gm M)_m$ at $m\in M$ we set
\[ u^O_f := \frac{ \text{d} }{ \text{d}t} \bigg{\vert}_{t=0} \tau_t  \circ \Phi_t f \, \in T_f(O_\gm M),\]
where $ \tau_t $ is the parallel transport in the bundle $ \on{Bas}(E_{\bar{\Phi}_t(m)}) \to \on{Met}(E_{\bar{\Phi}_t(m)}) $ along the curve:
\[ \gamma (s) := \Phi_{t-s} \left( (V_+)_{\bar{\Phi}_s(m)} \right)  \in \on{Met}(E_{\bar{\Phi}_t(m)}) ,\quad s\in [0,t]. \]
Let now $\Sigma$ be a representation of $O(r,s)\times O(p-r,q-s)$ and $b$ an equivariant function on $O_\gm M$ valued in $\Sigma$. The corresponding \textit{generalised metric Lie derivative} $\mathcal{L}^\gm_u$ in direction of $u \in \Gamma(E)$ is defined as:
\[ \mathcal{L}^{\gm}_u b := u^O b.\]
\end{defn}

\begin{remark}
  Due to the naturality (cf.\ Remark \ref{rem:natural}) of the parallel transport $\tau(t)$ we obtain the following equivalent definition of $u^O$:
  \begin{equation}\label{uodef2}
  u^O_f := \frac{ \text{d} }{ \text{d} t } \bigg{\vert}_{t=0} \Phi_{t} \circ \tau_t (f),
  \end{equation}
  where now $\tau_t\colon \on{Bas}_{V_+}(E_m)\to \on{Bas}_{\Phi_{-t}(V_+)}(E_m)$ is the parallel transport in $\on{Bas}(E_m)\to\on{Met}(E_m)$ along the curve \[\gamma(s)= \Phi_{-s} ((V_+)_{\Phi_s(m)})=(\Phi_{-s} V_+)_m \in\on{Met}(E_m),\quad s\in [0,t].\]
\end{remark}

Consider now the special case when $b$ corresponds to a section of $E$, i.e.\ $\Sigma$ is the representation
\[(\text{vector}\otimes\text{trivial})\oplus(\text{trivial}\otimes\text{vector})\] of the group $O(r,s)\times O(p-r,q-s)$, corresponding to the decomposition $E=V_+\oplus V_-$. In this case, one can derive the following more explicit formula for the generalised metric Lie derivative.
\begin{theorem}\label{genmetlierezyfull}
For any $ u,v \in \Gamma(E) $ we have $\mathcal{L}^\gm_u v= [u,v_+]_++[u,v_-]_-$.
\end{theorem}
\begin{proof}
We will use formula \eqref{uodef2}.
Let $ \tilde{v}$ be the equivariant function corresponding to $v\in\Gamma(E)$ 
and let $p_{V_\pm}\colon E\to V_\pm$ and $i_{V_\pm}\colon V_\pm\to E$ be the orthogonal projection and inclusion, respectively. Finally, in accordance with the notation before, we will use $\Phi$/$\tau$ to denote both the flow/transport on the bundle of frames and on the vector bundle. For instance, since the parallel transport (as well as its inverse) coincides to the linear order with the orthogonal projection, we have 
\[(\tau_\epsilon)^{-1}=p_{V_\pm}\circ i_{\Phi_{-\epsilon}(V_\pm)}+\mathcal O(\epsilon^2)\colon \Phi_{-\epsilon}(V_\pm)\to V_\pm,\]
or in other words (omitting now the inclusions)
\[(\tau_\epsilon)^{-1}=p_{V_+}\circ p_{\Phi_{-\epsilon}(V_+)}+p_{V_-}\circ p_{\Phi_{-\epsilon}(V_-)}+\mathcal O(\epsilon^2)\colon E\to E.\]
For any $V_\pm$-adapted frame $f=\{e_\alpha\}$ at $m\in M$ one then has (we write the dependence on points in parentheses):
\begin{align*}
[ u^O \tilde{v} ]^\alpha (f) e_\alpha &= \left[ \frac{ \text{d} }{ \text{d} t} \bigg{\vert}_{t=0} \tilde{v}^\alpha (\Phi_{t} \circ \tau_t f ) \right] e_\alpha = \frac{ \text{d} }{ \text{d} t} \bigg{\vert}_{t=0} \tilde{v}^\alpha (\Phi_{t} \circ \tau_t f ) \;(\tau_t)^{-1} \circ \Phi_{-t} \circ \Phi_t \circ \tau_t (e_\alpha) \\
&= \frac{ \text{d} }{ \text{d} t} \bigg{\vert}_{t=0} (\tau_{t})^{-1} \circ \Phi_{-t} \left[  v ( \bar{\Phi}_t(m)) \right] \\
&=\frac{ \text{d} }{ \text{d} t} \bigg{\vert}_{t=0} \left(p_{V_+}\circ p_{\Phi_{-t}(V_+)}\circ \Phi_{-t} \left[  v ( \bar{\Phi}_t(m)) \right]+ p_{V_-}\circ p_{\Phi_{-t}(V_-)}\circ \Phi_{-t} \left[  v ( \bar{\Phi}_t(m)) \right]\right)\\
&=\frac{ \text{d} }{ \text{d} t} \bigg{\vert}_{t=0} \left(p_{V_+}\circ \Phi_{-t} \circ p_{V_+} \left[  v ( \bar{\Phi}_t(m)) \right]+ p_{V_-}\circ \Phi_{-t}\circ p_{V_-} \left[  v ( \bar{\Phi}_t(m)) \right]\right)\\
&=\left[ \frac{ \text{d} }{ \text{d} t} \bigg{\vert}_{t=0}  \Phi_{-t} \,   v_+ ( \bar{\Phi}_t(m)) \right]_++\left[ \frac{ \text{d} }{ \text{d} t} \bigg{\vert}_{t=0}  \Phi_{-t} \,   v_- ( \bar{\Phi}_t(m)) \right]_- = ([u,v_+]_++[u,v_-]_-)(m),
\end{align*}
where in the last step we used \eqref{auto}.
\end{proof}

\begin{remark}
  The bundle $O_\gm M$ is a product of bundles $\on{Fr}(V_+)$ and $\on{Fr}(V_-)$ of orthonormal frames of $V_+$ and $V_-$, respectively. From the above discussion it follows that the vector field $u^O$ in fact arises as the sum of the vector fields on the individual bundles $\on{Fr}(V_\pm)$. Instead of working with bundle $O_\gm M$ one can thus equivalently work with the pair $\on{Fr}(V_+)$, $\on{Fr}(V_-)$ and study generalised metric Lie derivatives of sections of their associated vector bundles.
\end{remark}

\begin{cor}\label{genlielc}
Let $D$ be a generalised Levi-Civita connection for a generalised metric $V_+$ on $E\to M$. For $u,v\in\Gamma(E)$ and for any local adapted frame $\{e_\alpha\}=\{\{e_a\},\{e_{\dot c}\}\}$ of $V_+\oplus V_-$ we have
\[ (\mathcal{L}^\gm_u v)^a = u^\alpha D_\alpha v^a-2v_bD^{[b}u^{a]},\qquad (\mathcal{L}^\gm_u v)^{\dot c} = u^\alpha D_\alpha v^{\dot c}-2v_{\dot e}D^{[\dot e}u^{\dot c]}.\]
\end{cor}
\begin{proof}
Using the properties of the generalised Levi-Civita connection we calculate
\begin{align*}
(\mathcal L^\gm_uv)_\pm&=[u,v_\pm]_\pm=(D_uv_\pm-D_{v_\pm}u+\langle Du,v_\pm\rangle)_\pm=D_uv_\pm-D_{v_\pm}u_\pm+\langle Du_\pm,v_\pm\rangle_\pm.\qedhere
\end{align*}
\end{proof}
Again, we can restate the above Corollary in terms of the invariant vector fields on $O_\gm M$ as follows:
\begin{cor}\label{rem:gendifference}
  Let $D$ be a generalised Levi-Civita connection for a generalised metric $V_+$ on $E\to M$. For any $u\in \Gamma(E)$ we then have
  \begin{equation}\label{genmetric-vs-lc}
    u^O=u^{LC}+ 2 \,\Pi\circ\on{antisym}(Du),
  \end{equation}
  where the last term is understood in terms of the (infinitesimal) action of $O(r,s)\times O(p-r,q-s)$ on $O_\gm M$, with
  \[\Pi\colon \Lambda^2 E\to \Lambda^2 V_+\oplus \Lambda^2 V_-\]
  the orthogonal projection.
  \end{cor}
  We conclude this Subsection by examining the relation between the generalised metric Lie derivative w.r.t.\ the bracket on the Courant algebroid and the commutator of two generalised metric Lie derivatives (for the classical analogue see Proposition 18 in \cite{bg}).
\begin{prop}\label{genlielc2}
  For any $u,v,w\in\Gamma(E)$ we have $\mathcal L^\gm_{[u,v]}w=[\mathcal L^\gm_u,\mathcal L^\gm_v]w-\tfrac14[\mathcal L_u\gm,\mathcal L_v\gm]w$.
\end{prop}
\begin{proof}
Using the Jacobi identity and Lemma \ref{lem:killing} we calculate
\begin{align*}
\mathcal{L}^\gm_{[u,v]}w_\pm - [\mathcal{L}^\gm_u ,\mathcal{L}^\gm_v]w_\pm &= [[u,v],w_\pm]_\pm-[u,[v,w_\pm]_\pm]_\pm+[v,[u,w_\pm]_\pm]_\pm\\
&=  [u,[v,w_\pm]]_\pm-[v,[u,w_\pm]]_\pm-[u,[v,w_\pm]_\pm]_\pm+[v,[u,w_\pm]_\pm]_\pm\\
&=  [u,  [v,w_\pm]_\mp ]_\pm - [v , [u,w_\pm]_\mp ]_\pm =  -\tfrac{1}{4} \left[ \mathcal{L}_u \mathcal{G} , \mathcal{L}_v \mathcal{G} \right] w_\pm. \qedhere
\end{align*}
\end{proof}
\begin{remark}
  The Proposition can be restated in terms of vector fields on $O_\gm M$ as
  \begin{equation}\label{curv}
    [u,v]^O=[u^O,v^O]-\tfrac14[\mathcal L_u\gm,\mathcal L_v\gm],
  \end{equation}
  where the last term is understood in terms of the corresponding fundamental vector field on $O_\gm M$. Recalling that the generalised metric Lie derivative is constructed using the connection $\nabla^{\text{bas}}$ on the space of frames (over a given point $m\in M$), formula \eqref{curv} expresses the fact that $\nabla^{\text{bas}}$ is not flat; the second term on the RHS of \eqref{curv} corresponds precisely to its curvature. For more details (including the derivation of the curvature) we refer the reader to \cite{ksv}.
\end{remark}

\subsection{Generalised Lie derivative of pinor and spinor fields}
\begin{defn}
  A \emph{$\smash{\text{pin}^+}$ structure} on a Courant algebroid $E\to M$ is the choice of a generalised metric $V_+\subset E$ of signature $(r,s)$ together with a lift of the $O(r,s)\times O(p-r,q-s)$-structure to a $\pin(r,s)\times O(p-r,q-s)$-structure. We will call the associated principal bundle $P_\gm M\to M$. A \emph{pinor field} is an equivariant function on $P_\gm M$ valued in the tensor product of the pinor representation of $\pin(r,s)$ and the trivial representation of $O(p-r,q-s)$.
\end{defn}

Using the analogue of Lemma \ref{lemma:lift} we can again lift any $u^O$ to a vector field on $P_\gm M$:
\begin{defn}
Let $E\to M$ be a Courant algebroid with a $\smash{\text{pin}^+}$ structure. For a section $ u \in \Gamma(E) $ we define $ u^P \in \mathfrak{X}(P_\gm M) $ to be the lift of $ u^O \in \mathfrak{X}(O_\gm M) $. Let $\Sigma$ be a representation of $\pin(r,s)\times O(p-r,q-s)$ and $\psi\colon P_\gm M\to \Sigma$ an equivariant function. We define its \emph{generalised Lie derivative} as
\[ \mathcal{L}_u \psi := u^P \psi. \]
\end{defn}

As an immediate corollary of (\ref{genmetric-vs-lc}) and Proposition \ref{genlielc2} we obtain:
\begin{theorem}\label{ggspin}
  Let $E \to M $ be a Courant algebroid with a pin$^+$ structure and a generalised Levi-Civita connection $D$. For any $u\in \Gamma(E)$ and any pinor field $\psi$ we have
  \[\mathcal L_u \psi= D_u\psi+\tfrac12(D_a u_b)\gamma^{ab}\psi.\]
  For any pair $u,v\in\Gamma(E)$ we obtain
  \[ \mathcal{L}_{[u,v] } \psi = [ \mathcal{L}_u, \mathcal{L}_v ] \psi + \tfrac{1}{16} [ \mathcal{L}_u \mathcal{G} , \mathcal{L}_v \mathcal{G}]_{ab} \gamma^{ab} \psi. \]
\end{theorem}

Finally, we address the case of spinors.
\begin{defn}
  A \emph{spin${}^+$ structure} on a Courant algebroid $E\to M$ consists of a choice of a generalised metric with an orientation of the bundle $V_+$ (producing an $SO(r,s)\times O(p-r,q-s)$-principal bundle $SO_\gm M$) together with a lift of the $SO(r,s)\times O(p-r,q-s)$-structure to a $\on{Spin}(r,s)\times O(p-r,q-s)$-one.
\end{defn}
\begin{remark}
  This structure appears naturally in type I supergravity in 10 dimensions (cf.\ \cite{csw}); in type II the natural structure is instead given by the lift to a $\on{Spin}(r,s)\times \on{Spin}(p-r,q-s)$ structure.\footnote{More precisely, in type II supergravity one requires specifically a $\on{Spin}(9,1)\times \on{Spin}(1,9)$ structure.} The extension of the results below to the latter case is obvious.
\end{remark}
\begin{defn}
  Let $E\to M$ be a Courant algebroid with a spin${}^+$ structure. For any section $u\in\Gamma(E)$ we take $u^{Sp}$ to be the lift of $u^O|_{SO_\gm M}\in\mathfrak X(SO_\gm M)$ to the principal $\on{Spin}(r,s)\times O(p-r,q-s)$-bundle. Let $\Sigma$ be a representation of $\on{Spin}(r,s)\times O(p-r,q-s)$ and $\Xi$ a corresponding equivariant function. We then define the \emph{generalised Lie derivative} by
  \[\mathcal L_u\Xi:=u^{Sp}\,\Xi.\]
\end{defn}
The two particular representations appearing in $N=1$ supergravity are \[\text{spinor}\otimes\text{trivial},\qquad \text{spinor}\otimes\text{vector},\]
in which cases we have for the corresponding sections
\begin{equation}\label{spinor}
  \mathcal L_u \psi= D_u\psi+\tfrac12(D_a u_b)\gamma^{ab}\psi, \qquad (\mathcal L_u \psi)^{\dot c}= D_u\psi^{\dot c}+\tfrac12(D_a u_b)\gamma^{ab}\psi^{\dot c}+2(D^{[\dot c}u^{\dot d]})\psi_{\dot d},
\end{equation}
respectively (as usual we keep the vector index explicit in the latter case).

\setstretch{1.15}

\end{document}